\documentclass{article}
\usepackage{amsfonts,latexsym,hyperref} % label
\newcounter{satznum}
\newtheorem{theorem}{Theorem}[satznum]

\newtheorem{lemma}[theorem]{Lemma}
\newtheorem{corollary}[theorem]{Corollary}

\newenvironment{acknowledgement}
 {\begin{trivlist}\item[]{\bf Acknowledgement.}}
 {\end{trivlist}}
\newenvironment{remark}
 {\begin{trivlist}\item[]{\bf Remark.}}
 {\end{trivlist}}
\newenvironment{remarks}
 {\begin{trivlist}\item[]{\bf Remarks.}}
 {\end{trivlist}}
\newenvironment{example}
 {\begin{trivlist}\item[]{\bf Example.}}
 {\end{trivlist}}

\newenvironment{proof}
 {\begin{trivlist}\item[]{\bf Proof.}}
 {\end{trivlist}}
{\makeatletter
\gdef\cz{{\mathbb C}} % complex numbers
\gdef\nz{{\mathbb N}} % positive integers
\gdef\rz{{\mathbb R}} % real numbers
}
\newcounter{todocounter}

\makeatletter
\def\@MRExtract#1 #2!{#1}
\newcommand{\MR}[1]{% we need to strip the "(...)"
  \xdef\@MRSTRIP{\@MRExtract#1 !}
  \href{http://www.ams.org/mathscinet-getitem?mr=\@MRSTRIP}{MR\@MRSTRIP}}
\makeatother

\gdef\pr{{\mathbb P}}
\gdef\me{{\mathbb E}}
\begin{document}
   \section*{ON THE EXTERNAL BRANCHES OF COALESCENTS WITH MULTIPLE
   COLLISIONS} % WITH AN EMPHASIS ON THE BOLTHAUSEN--SZNITMAN COALESCENT}
   {\sc J.-S. Dhersin}\footnote{Department de Mathematique,  Insitut
   Galil\'ee, Universit\'e Paris 13, Avenue Jean Baptiste Cl\'ement, 93430 Villetaneuse, France,
   E-mail: dhersin@math.univ-paris13.fr}
   and {\sc M. M\"ohle}\footnote{Mathematisches Institut, Eberhard Karls
   Universit\"at T\"ubingen, Auf der Morgenstelle 10, 72076 T\"ubingen,
   Germany, E-mail address: martin.moehle@uni-tuebingen.de}
   \hfill\today
%      \begin{center}
%      \today
%      \ \\
%      {\tt Preliminary Version}\\
%      {\tt - not ready for publication -}
%   \end{center}
%
\begin{abstract}
   A recursion for the joint moments of the external branch lengths
   for coalescents with multiple collisions ($\Lambda$-coalescents) is
   provided. This recursion is used to derive asymptotic results as
   the sample size $n$ tends to infinity for the joint moments of the
   external branch lengths and for the moments of the total
   external branch length of the Bolthausen--Sznitman coalescent.
   These asymptotic results are based on a differential equation approach,
   which is as well useful to obtain exact solutions for the joint moments of
   the external branch lengths for the Bolthausen--Sznitman coalescent.
   The results for example show that the lengths of two randomly chosen
   external branches are positively correlated for the Bolthausen--Sznitman
   coalescent, whereas they are negatively correlated for the
   Kingman coalescent provided that $n\ge 4$.

   \vspace{1mm}

   \noindent Keywords: Asymptotic expansions;
   Bolthausen--Sznitman coalescent; external branches; joint moments;
   Kingman coalescent; multiple collisions

   \vspace{1mm}

   \noindent 2010 Mathematics Subject Classification:
            Primary 60J25; % Continuous-time Markov processes on general state space
                    34E05; % Asymptotic expansions
                    60C05  % Combinatorial probability
%                    60F05  % Central limit and other weak theorems
            Secondary
%                      05C05; % Trees
                      60J85; % Applications of branching processes
%                      92D10; % Genetics
                      92D15;  % Problems related to evolution
                      92D25   % Population Dynamics (general)
\end{abstract}
\subsection{Introduction and main results} \label{intro}
Let $\Pi=(\Pi_t)_{t\ge 0}$ be a coalescent process with multiple collisions
($\Lambda$-coalescent). For fundamental information on
$\Lambda$-coalescents we refer the reader to \cite{pitman} and
\cite{sagitov}. For $n\in\nz:=\{1,2,\ldots\}$ we denote with
$\Pi^{(n)}=(\Pi^{(n)}_t)_{t\ge 0}$ the coalescent process restricted
to $[n]:=\{1,\ldots,n\}$. Note that $\Pi^{(n)}$ is Markovian with state
space ${\cal E}_n$, the set of all equivalence relations (partitions) on
$[n]$. For $\xi\in{\cal E}_n$ we write $|\xi|$ for the number of equivalence
classes (blocks) of $\xi$. For $m\in\{1,\ldots,n-1\}$ let $g_{nm}$ be the
rate at which the block counting process $N^{(n)}:=(N^{(n)}_t)_{t\ge 0}
:=(|\Pi^{(n)}_t|)_{t\ge 0}$ jumps at its first jump time from $n$ to $m$.
It is well known (see, for example, \cite[Eq.~(13)]{moehleewens}) that
\begin{equation} \label{rates}
   g_{nm}\ =\ {n\choose{m-1}}\int_{[0,1]}\,x^{n-m-1}(1-x)^{m-1}\,\Lambda({\rm d}x)
\end{equation}
for all $n,m\in\nz$ with $m<n$.
We furthermore introduce the total rates
\begin{equation} \label{totalrates}
   g_n\ :=\ \sum_{m=1}^{n-1}g_{nm}
   \ =\ \int_{[0,1]}\frac{1-(1-x)^n-nx(1-x)^{n-1}}{x^2}\,\Lambda({\rm d}x),\qquad n\in\nz.
\end{equation}
We are interested in the external branches of the restricted
coalescent process $\Pi^{(n)}$. More precisely, for $n\in\nz$ and
$i\in\{1,\ldots,n\}$ let $\tau_{n,i}:=\inf\{t>0:\mbox{$\{i\}$ is a
singleton block of $\Pi_t^{(n)}$}\}$ denote the length of the
$i$th external branch of the restricted coalescent $\Pi^{(n)}$.
Note that $\tau_{1,1}=0$. Our first main result (Theorem
\ref{theorem11}) provides a general recursion for the joint
moments
\begin{equation} \label{moments}
   \mu_n(k_1,\ldots,k_j)
   \ :=\ \me(\tau_{n,1}^{k_1}\cdots\tau_{n,j}^{k_j}),
   \qquad j\in\{1,\ldots,n\}, k_1,\ldots,k_j\in\nz_0:=\{0,1,\ldots\},
\end{equation}
of the external branch lengths. The proof of Theorem \ref{theorem11} is
provided in Section \ref{proofs}.
\begin{theorem}[Recursion for the joint moments of the external branch lengths]
\label{theorem11}
   \ \\
   For all $n\ge 2$, $j\in\{1,\ldots,n\}$ and
   $k=(k_1,\ldots,k_j)\in\nz^j$ the joint moments
   $\mu_n(k):=\me(\tau_{n,1}^{k_1}\cdots\tau_{n,j}^{k_j})$
   of the lengths $\tau_{n,1},\ldots,\tau_{n,n}$ of the external branches
   of a $\Lambda$-coalescent $\Pi^{(n)}$ satisfy the recursion
   \begin{equation} \label{generalrec}
      \mu_n(k)
      \ = \ \frac{1}{g_n}\sum_{i=1}^j k_i\,
            \mu_n(k-e_i)
            + \sum_{m=j+1}^{n-1}p_{nm}\frac{(m-1)_j}{(n)_j}\,
            \mu_m(k),
   \end{equation}
   where $e_i$, $i\in\{1,\ldots,j\}$, denotes the $i$th
   unit vector in $\rz^j$, $p_{nm}:=g_{nm}/g_n$ and $g_{nm}$ and $g_n$ are
   defined via (\ref{rates}) and (\ref{totalrates}).
\end{theorem}
\begin{remarks}
   The recursion (\ref{generalrec}) works as follows. Let
   us call $d:=k_1+\cdots+k_j$ the order (or degree) of the moment
   $\mu_n(k_1,\ldots,k_j)$. Provided that all the moments of order
   $d-1$ are already known, (\ref{generalrec}) is a recursion
   on $n$ for the joint moments of order $d$, which can be solved
   iteratively. So one starts with $d=1$ (and hence $j=1$), in which
   case (\ref{generalrec}) reduces to
   $\mu_n(1)=1/g_n+\sum_{m=2}^{n-1}p_{nm}((m-1)/n)\mu_m(1)$,
   $n\ge 2$. Since $\mu_2(1)=\me(\tau_{2,1})=1/g_2
   =1/\Lambda([0,1])$, this recursion determines the moments of order
   $1$ completely. Now choose $d=2$ in (\ref{generalrec}) leading
   to a recursion for the second order moments. Iteratively, one can
   move to higher orders.
   Note that for $j=2$ and $k_1=k_2=1$ the recursion (\ref{generalrec})
   reduces to
   \begin{equation} \label{tau12rec}
      \me(\tau_{n,1}\tau_{n,2})
      \ =\ \frac{2}{g_n}\me(\tau_{n,1}) +
      \sum_{m=2}^{n-1} p_{nm}\frac{(m-1)_2}{(n)_2}
      \me(\tau_{m,1}\tau_{m,2}),
      \quad n\in\{2,3,\ldots\}.
   \end{equation}
%         Note that $\me(\tau_{2,1}\tau_{2,2})=2/g_2^2$.
%         Provided
%         that $\me(\tau_{n,1})$ is known, (\ref{tau12rec}) is a recursion
%         on $n$ for $\me(\tau_{n,1}\tau_{n,2})$.
%   \end{enumerate}
\end{remarks}
Note that Theorem \ref{theorem11} holds for arbitrary
$\Lambda$-coalescents. For particular $\Lambda$-coalescents the
recursion (\ref{generalrec}) can be used to derive exact solutions
and asymptotic expansions for the joint moments of the lengths of
the external branches. In the following we briefly discuss the
star-shaped coalescent and the Kingman coalescent. Afterwards we
intensively study the % external branch lengths of the
Bolthausen--Sznitman coalescent. For related results on external
branches for beta-coalescents we refer the reader to
\cite{dhersinfreundsirijegousseyuan}, \cite{dhersinyuan} and
\cite{kerstingstanciuwakolbinger}.
\begin{example} (Star-shaped coalescent)
   For the star-shaped coalescent, where $\Lambda$ is the Dirac measure
   at $1$, the time $T_n$ of the first jump of $\Pi^{(n)}$ is
   exponentially distributed with parameter $g_n=1$, $n\in\{2,3,\ldots\}$.
   Furthermore, $p_{nm}=\delta_{m1}$ for $n,m\in\nz$ with $m<n$. Thus,
   (\ref{generalrec}) reduces to $\mu_n(k)=\sum_{i=1}^j k_i\,\mu_n(k-e_i)$
   with solution $\mu_n(k)=(k_1+\cdots+k_j)!$, which is obviously correct,
   since $\tau_{n,i}=T_n$ for all $i\in\{1,\ldots,n\}$ and, therefore,
   $\mu_n(k)=\me(T_n^{k_1+\cdots+k_j})=(k_1+\cdots+k_j)!$,
   $n\ge 2$, $j\in\{1,\ldots,n\}$, $k_1,\ldots,k_j\in\nz$.
\end{example}
\begin{example} (Kingman coalescent)
   For the Kingman coalescent \cite{kingman}, where $\Lambda$ is the
   Dirac measure at $0$, the time $T_n$ of the first jump of $\Pi^{(n)}$
   is exponentially distributed with parameter $g_n=n(n-1)/2$,
   $n\in\{2,3,\ldots\}$. Furthermore, $p_{nm}=\delta_{m,n-1}$ for $m,n\in\nz$
   with $m<n$. Caliebe et al. \cite[Theorem~1]{caliebeneiningerkrawczakroesler}
   verified that $n\tau_{n,1}\to Z$ in distribution as $n\to\infty$, where $Z$
   has density $x\mapsto 8/(2+x)^3$, $x\ge 0$. Janson and Kersting
   \cite[Theorem~1]{jansonkersting} showed that the total external branch
   length $L_n^{external}:=\sum_{i=1}^n\tau_{n,i}$ satisfies
   $(1/2)\sqrt{n/(\log n)}(L_n^{external}-2)\to N(0,1)$ in distribution as
   $n\to\infty$. We are instead interested here in the moments of
   $\tau_{n,1}$. The recursion (\ref{generalrec}) for $j=1$ reduces to
   $$
   \mu_n(k)
   \ =\ \frac{2k}{n(n-1)}\,\mu_n(k-1)
   + \frac{n-2}{n}\,\mu_{n-1}(k),\qquad n\in\{2,3,\ldots\},
   k\in\nz.
   $$
   Rewriting this recursion in terms of $a_n(k):=n(n-1)\mu_n(k)$
   yields $a_n(k)=2k\,\mu_n(k-1)+a_{n-1}(k)$, $n\in\{2,3,\ldots\}$,
   $k\in\nz$, with solution $a_n(k)=2k\sum_{m=2}^n\mu_m(k-1)$.
   Thus,
   $$
   \mu_n(k)\ =\ \frac{2k}{n(n-1)}\sum_{m=2}^n \mu_m(k-1),
   \qquad n\in\{2,3,\ldots\}, k\in\nz.
   $$
   The first two moments are therefore
   $\me(\tau_{n,1})=\mu_n(1)=2/(n(n-1))\sum_{m=2}^n 1=2/n$ and
   $$
   \me(\tau_{n,1}^2)\ =\ \mu_n(2)
   \ =\ \frac{4}{n(n-1)}\sum_{m=2}^n \frac{2}{m}
   \ =\ \frac{8(h_n-1)}{n(n-1)}
   \ =\ 8\frac{\log n}{n^2} + \frac{8(\gamma-1)}{n^2} + O\bigg(\frac{\log n}{n^3}\bigg),
   $$
   where $\gamma\approx 0.577216$ denotes the Euler constant and
   $h_n:=\sum_{i=1}^n 1/i$ the $n$-th harmonic number, $n\in\nz$. Note that
   these results are in agreement with those of Caliebe et al.
   \cite[Eq.~(2)]{caliebeneiningerkrawczakroesler} and Janson and
   Kersting \cite[p.~2205]{jansonkersting}. For the third moment we obtain
   $$
   \mu_n(3)
   \ =\ \frac{6}{n(n-1)}\sum_{m=2}^n \frac{8(h_m-1)}{m(m-1)}
   \ =\ \frac{48}{n(n-1)}\sum_{m=2}^n \frac{h_m-1}{m(m-1)}.
   $$
   Since $h_{m+1}-h_m=1/(m+1)$, the last sum simplifies considerably to
   \begin{eqnarray*}
      \sum_{m=2}^n\frac{h_m-1}{m(m-1)}
      & = & \sum_{m=2}^n\bigg(\frac{h_m}{m-1} - \frac{h_m}{m} - \frac{1}{m(m-1)}\bigg)
      \ = \ \sum_{m=1}^{n-1} \frac{h_{m+1}}{m} - \sum_{m=2}^n \frac{h_m}{m} - \bigg(1 - \frac{1}{n}\bigg)\\
      & = & h_2 + \sum_{m=2}^{n-1}\frac{1}{m(m+1)}
            - \frac{h_n}{n} - 1 + \frac{1}{n}
%      \ = \ \frac{1}{2} + \sum_{m=2}^{n-1}\frac{1}{m(m+1)} - \frac{h_n}{n} + \frac{1}{n}
%      \ = \ \frac{3}{2} + \frac{1}{2} - \frac{1}{n} - \frac{h_n}{n} - 1 + \frac{1}{n}
      \ = \ 1 - \frac{h_n}{n},
   \end{eqnarray*}
   Thus, the third moment of $\tau_{n,1}$ is
   $$
   \me(\tau_{n,1}^3)\ =\ \mu_n(3)\ =\ \frac{48}{n(n-1)}\bigg(1-\frac{h_n}{n}\bigg)
   \ =\ \frac{48}{n^2} - 48\frac{\log n}{n^3} + O\bigg(\frac{1}{n^3}\bigg).
   $$
   For the fourth moment we obtain
   $$
   \me(\tau_{n,1}^4)
   \ =\ \mu_n(4)
   \ =\ \frac{8}{n(n-1)}\sum_{m=2}^n \mu_m(3)
   \ =\ \frac{384}{n(n-1)}\sum_{m=2}^n \frac{1-h_m/m}{m(m-1)},
   $$
   a formula which does not seem to simplify much further. One may also
   introduce the generating functions $g_k(t):=\sum_{n=2}^\infty\mu_n(k)t^n$,
   $k\in\nz$, $|t|<1$. For all $k\ge 2$ we have
   \begin{eqnarray*}
      t^2g_k''(t)
      & = & \sum_{n=2}^\infty n(n-1)\mu_n(k)t^n
      \ = \ \sum_{n=2}^\infty 2k\sum_{m=2}^n \mu_m(k-1) t^n\\
      & = & 2k\sum_{m=2}^\infty \mu_m(k-1)t^m\sum_{n=m}^\infty t^{n-m}
      \ = \ \frac{2k}{1-t}g_{k-1}(t),
   \end{eqnarray*}
   so these generating functions satisfy the recursion
   $$
   g_k(t)\ =\ 2k\int_0^t\int_0^s \frac{g_{k-1}(u)}{u^2(1-u)}
   \,{\rm d}u\,{\rm d}s,\qquad k\ge 2, 0\le t<1,
   $$
   with initial function $g_1(t)=\sum_{n=2}^\infty
   (2/n)t^n=-2t-2\log(1-t)$. Using this recursion, $g_k(t)$ can be
   computed iteratively, however, the expressions become quite involved with
   increasing $k$. For example, $g_2(t)=8t-4(1-t)\log^2(1-t)-8(1-t){\rm Li}_2(t)$,
   $|t|<1$, where ${\rm Li}_2(t):=-\int_0^t (\log(1-x))/x\,{\rm d}x=
   \sum_{k=1}^\infty t^k/k^2$ denotes the dilogarithm function.
   In principle
   higher order moments and as well joint moments can be calculated
   analogously, however the expressions become more and more nasty with
   increasing order. In the following we exemplary derive an exact formula
   for $\mu_n(1,1)=\me(\tau_{n,1}\tau_{n,2})$. The recursion
   (\ref{generalrec}) for $j=2$ and $k_1=k_2=1$ reduces to (see (\ref{tau12rec}))
   $$
   \mu_n(1,1)
   \ =\ \frac{2}{g_n}\mu_n(1) + \frac{(n-2)_2}{(n)_2}\,\mu_{n-1}(1,1)
   \ = \frac{8}{n^2(n-1)} + \frac{(n-2)(n-3)}{n(n-1)}\,\mu_{n-1}(1,1),
   \qquad n\ge 2.
   $$
   It is readily checked by induction on $n$ that this recursion is solved
   by $\mu_2(1,1)=2$ and
   $$
   \mu_n(1,1)\ =\ \frac{4(n^2-5n+4h_n)}{n(n-1)^2(n-2)},\qquad n\in\{3,4,\ldots\}.
   $$
   In particular, $\mu_n(1,1)=4/n^2 - 4/n^3+ O((\log n)/n^4)$, $n\to\infty$.
   Moreover, ${\rm Cov}(\tau_{n,1},\tau_{n,2})=
   \mu_n(1,1)-(\mu_n(1))^2=4(n^2-5n+4h_n)/(n(n-1)^2(n-2))-4/n^2<0$ for all
   $n\ge 4$. Thus, for the Kingman coalescent, the lengths
   of two randomly chosen external branches are (slightly) negatively
   correlated for all $n\ge 4$.
   We have used the derived formulas to compute the following table.
   \begin{center}
      \begin{tabular}{|r|l|l|l|}
         \hline
         $n$ & $\mu_n(1)=\me(\tau_{n,1})$ & $\mu_n(1,1)=\me(\tau_{n,1}\tau_{n,2})$ & ${\rm Cov}(\tau_{n,1},\tau_{n,2})$\\
         \hline
         $2$ & $1$ & $2$ & $1$\\
         $3$ & $0.666667$ & $0.444444$ & $0$\\
         $4$ & $0.5$ & $0.240741$ & $-0.009259$\\
         $5$ & $0.4$ & $0.152222$ & $-0.007778$\\
         $10$ & $0.2$ & $0.038096$ & $-0.001904$\\
         $100$ & $0.02$ & $0.000396$ & $-0.000004$\\
%         $200$ & $0.010000$ & & \\
%         $300$ & $\approx 0.006667$ & & \\
%         $400$ & $0.005$ & & \\
%         $500$ & $0.004$ & & \\
%         $1000$ & $0.002$ & &\\
         \hline
         $n\to\infty$ & $\frac{2}{n}$ & $\frac{4}{n^2} - \frac{4}{n^3}+O(\frac{\log n}{n^4})$ & $-\frac{4}{n^3}+O(\frac{\log n}{n^4})$\\
         \hline
      \end{tabular}\\

      \vspace{1mm}

      Table 1: Covariance of $\tau_{n,1}$ and $\tau_{n,2}$ for the
      Kingman coalescent
   \end{center}
\end{example}
In the following we focus on the Bolthausen--Sznitman coalescent
\cite{bolthausensznitman}, where $\Lambda$ is the uniform
distribution on $[0,1]$. Our second main result (Theorem
\ref{theorem12}) provides the asymptotics of all the joint moments
of the external branch lengths for the Bolthausen--Sznitman
coalescent.
\begin{theorem}[Asymptotics of the joint moments of the external
branch lengths] \label{theorem12}
   \ \\
   For the Bolthausen--Sznitman coalescent, the joint moments
   $\mu_n(k):=\me(\tau_{n,1}^{k_1}\cdots\tau_{n,j}^{k_j})$,
   $j\in\nz$, $k=(k_1,\ldots,k_j)\in\nz_0^j$,
   of the lengths $\tau_{n,1},\ldots,\tau_{n,n}$ of the external
   branches satisfy
   \begin{equation} \label{bsasy}
      \mu_n(k)
      \ \sim\ \frac{k_1!\cdots k_j!}{\log^{k_1+\cdots+k_j}n},
      \qquad n\to\infty.
   \end{equation}
\end{theorem}
\begin{remark}
   For $j=2$ and $k_1=k_2=1$ Eq. (\ref{bsasy}) implies that
   $\me(\tau_{n,1}\tau_{n,2})
      =\mu_n(1,1)
      \sim 1/\log^2n\sim (\mu_n(1))^2$ as $n\to\infty$,
   which does not provide much information on the covariance
   ${\rm Cov}(\tau_{n,1},\tau_{n,2})=
   \mu_n(1,1) - (\mu_n(1))^2$.
   With some more effort
   (see Corollary \ref{corollary32} and the remark thereafter)
   exact solutions for $\me(\tau_{n,1})$ and $\me(\tau_{n,1}\tau_{n,2})$
   are obtained and it follows that $\tau_{n,1}$ and $\tau_{n,2}$
   are positively correlated for all $n\ge 2$, in contrast
   to the situation for the
   Kingman coalescent, where $\tau_{n,1}$ and $\tau_{n,2}$ are
   slightly negatively correlated for all $n\ge 4$.
\end{remark}
The following two corollaries are a direct
consequence of Theorem \ref{theorem12}.
\begin{corollary}[Weak limiting behavior of the external branch lengths]
\label{corollary13}
   \ \\
   For the Bolthausen--Sznitman coalescent,
   $(\log n)(\tau_{n,1},\ldots,\tau_{n,n},0,0,\ldots)\to
   (\tau_1,\tau_2,\ldots)$ in distribution as $n\to\infty$, where
   $\tau_1,\tau_2,\ldots$ are independent and all exponentially
   distributed with parameter $1$.
\end{corollary}
The following result concerns the asymptotics of the total
external branch length $L_n^{external}:=\sum_{i=1}^n\tau_{n,i}$ of
the Bolthausen--Sznitman coalescent.
\begin{corollary}[Asymptotics of the total external branch length]
\label{corollary14}
   \ \\
   Fix $k\in\nz$. For the Bolthausen--Sznitman coalescent, the $k$th moment
   of $L_n^{external}$ satisfies
   \begin{equation} \label{lnexternalkasy}
      \me((L_n^{external})^k)
      \ \sim\ \frac{n^k}{\log^kn},
%           \bigg(
%              1+O\bigg(\frac{1}{\log n}\bigg)
%           \bigg),
      \qquad n\to\infty.
   \end{equation}
In particular, $\frac{\log n}{n}L_n^{external}\to 1$ in
probability as $n\to\infty$.
\end{corollary}
The moments of $L_n^{external}$ do not provide much information on
the distributional limiting behavior of $L_n^{external}$ as
$n\to\infty$.
% Nevertheless, the expansions of the centered moments
% of $L_n$ and $L_n^{external}$ coincide, which supports (or at
% least does not contradict) the intuition that the distributional
% limiting behavior of $L_n^{external}$ should essentially coincide
% with that of $L_n$ (see, for example, \cite[Theorem
% 5.2]{drmotaiksanovmoehleroesler1}). Note that the total internal
% branch length $L_n^{internal}=L_n-L_n^{external}$ satisfies
% $$
% \frac{\log^2n}{n}\me(L_n^{internal})
% \ =\ 1+O\bigg(\frac{1}{\log n}\bigg)
% \ \to\ 1
% $$
% as $n\to\infty$.
   Let $L_n$ denote the total branch length (the sum of the lengths of
   all branches) of the Bolthausen--Sznitman
   $n$-coalescent.
   Kersting et al. \cite[Theorem 1.1]{kerstingpardosirijegousse} recently
   showed that the internal branch length $L_n^{internal}:=L_n-L_n^{external}$
   satisfies
   $$
   \frac{\log^2n}{n}L_n^{internal}
%   \ =\ 1+O\bigg(\frac{1}{\log n}\bigg)
   \ \to\ 1
   $$
   in probability. Combining this result with
   \cite[Theorem 5.2]{drmotaiksanovmoehleroesler1} it follows that
   (see \cite[Corollary 1.2]{kerstingpardosirijegousse})
   \begin{equation} \label{convergence}
      \frac{\log^2n}{n}L_n^{external} - \log n-\log\log n\ \to\ L-1
   \end{equation}
   in distribution as $n\to\infty$, where $L$ is a $1$-stable random
   variable with characteristic function
   $t\mapsto\exp(it\log|t|-\pi|t|/2)$, $t\in\rz$.
\begin{remark}
   The same scaling and, except for the additional shift $-1$ on
   the right hand side in (\ref{convergence}), the same limiting law
   as in (\ref{convergence}) is known for the number of cuts needed
   to isolate the root of a random recursive tree
   (\cite{drmotaiksanovmoehleroesler2}, \cite{iksanovmoehle}). Essentially
   the same scaling and convergence result has been obtained for random
   records and cuttings in binary search trees by Holmgren
   \cite[Theorem 1.1]{holmgren2} and more generally in split trees
   (Holmgren \cite[Theorem 1.1]{holmgren1} and \cite[Theorem 1.1]{holmgren3})
   introduced by Devroye \cite{devroye}. The logarithmic height of the
   involved trees seems to be one of the main sources for the occurrence
   of such scalings and of $1$-stable limiting laws.
   To the best of the authors knowledge the distributional limiting behavior of
   $L_n^{internal}$, properly centered and scaled,
   is so far unknown for the Bolthausen--Sznitman coalescent.
%   We conjecture that $L_n^{internal}$, properly centered and
%   scaled, is asymptotically normal.
\end{remark}
\subsection{Proof of Theorem \ref{theorem11}} \label{proofs}
\setcounter{theorem}{0}
% \begin{proof} (of Theorem \ref{theorem11})
   Let $T=T_n$ denote the time of the first jump of the block counting process
   $N^{(n)}$ and let $I=I_n$ denote the state of % the block counting process
   $N^{(n)}$ at its first jump. Note that $T$ and $I$ are independent, $T$ is
   exponentially
   distributed with parameter $g_n$ and $p_{nm}:=\pr(I=m)=g_{nm}/g_n$,
   $m\in\{1,\ldots,n-1\}$. % Moreover, $T$ and $I$ are independent.
   For $i\in\{1,\ldots,n\}$ and $h>0$ define
   $\tau_i':=\tau_{n,i}-h\wedge T$. By the Markov property,
   for $h\to 0$,
   \begin{eqnarray*}
      &   & \hspace{-15mm}
            \me(\tau_{n,1}^{k_1}\cdots\tau_{n,j}^{k_j}1_{\{T>h\}})
      \ = \ \me((\tau_1'+h)^{k_1}\cdots(\tau_j'+h)^{k_j}1_{\{T>h\}})\\
      & = & \me(\tau_{n,1}^{k_1}\cdots\tau_{n,j}^{k_j})\pr(T>h) +
            h\sum_{i=1}^j k_i
            \me(
               \tau_{n,1}^{k_1}\cdots\tau_{n,i-1}^{k_{i-1}}
               \tau_{n,i}^{k_i-1}
               \tau_{n,i+1}^{k_{i+1}}\cdots\tau_{n,j}^{k_j}
            ) + o(h).
   \end{eqnarray*}
   Also for $h\to 0$,
   $$
   \me(\tau_{n,1}^{k_1}\cdots\tau_{n,j}^{k_j}1_{\{T\le h\}})
   \ = \ \me((\tau_1'+T)^{k_1}\cdots(\tau_j'+T)^{k_j}1_{\{T\le h\}})
   \ = \ \me((\tau_1')^{k_1}\cdots(\tau_j')^{k_j}1_{\{T\le h\}}) + o(h).
   $$
   Now at time $T$ either the event
   $
   A:=\{\mbox{one of the individuals $1$ to $j$ is involved in the first collision}\}
   $
   occurs, in which case $\tau_i'=0$ for some $i\in\{1,\ldots,j\}$,
   and the above expectation vanishes since $k_1,\ldots,k_j>0$, or none of these
   $j$ individuals is involved in the first collision. Then, by the
   strong Markov property,
   $$
   \me((\tau_1')^{k_1}\cdots(\tau_j')^{k_j}1_{\{T\le h,I=m,A^c\}})
   \ =\ \me(\tau_{m,1}^{k_1}\cdots\tau_{m,j}^{k_j})\pr(T\le h,I=m,A^c),
   $$
   where $A^c$ denotes the complement of $A$. Adding both expectations yields
   \begin{eqnarray*}
      \me(\tau_{n,1}^{k_1}\cdots\tau_{n,j}^{k_j})
      & = & \me(\tau_{n,1}^{k_1}\cdots\tau_{n,j}^{k_j})\pr(T>h)
            +h\sum_{i=1}^j k_i\me(\tau_{n,1}^{k_1}\cdots\tau_{n,i-1}^{k_{i-1}}\tau_{n,i}^{k_i-1}\tau_{n,i+1}^{k_{i+1}}\cdots\tau_{n,j}^{k_j})\\
      &   & + \sum_{m=j+1}^{n-1}\me(\tau_{m,1}^{k_1}\cdots\tau_{m,j}^{k_j})
            \pr(T\le h)\pr(I=m)\frac{(m-1)_j}{(n)_j} + o(h).
   \end{eqnarray*}
   Collecting both terms involving $\me(\tau_{n,1}^{k_1}\cdots\tau_{n,j}^{k_j})$
   on the left hand side and letting $h\to 0$ gives the claim, since
   $\pr(T\le h)=1-e^{-g_nh}\sim g_nh$ as $h\to 0$.\hfill$\Box$
\subsection{Differential equations approach}
\setcounter{theorem}{0}
   A differential equations approach is provided, which is used in the
   proof of Theorem \ref{theorem12} given in the following Section
   \ref{furtherproofs}. This approach furthermore yields for example an
   exact expression for $\me(\tau_{n,1}\tau_{n,2})$ in terms of Stirling
   numbers (see Corollary \ref{corollary32}).
   Let $D:=\{z\in\cz\,:\,|z|<1\}$ denote the open unit disc in the complex plane.
   For $j\in\nz$ and $k=(k_1,\ldots,k_j)\in\nz_0^j$ define the generating function
   $$
   f_k(z)\ :=\ \sum_{n=j}^\infty \me(\tau_{n,1}^{k_1}\cdots\tau_{n,j}^{k_j})z^{n-1}
   \ =\ \sum_{n=j}^\infty a_n z^{n-1},\qquad z\in D,
   $$
   where, for $n\ge j$, we use the abbreviation
   $a_n:=\mu_n(k):=\me(\tau_{n,1}^{k_1}\cdots\tau_{n,j}^{k_j})$ for convenience.
   Note that, due to the natural coupling property of $n$-coalescents, the
   sequence $(a_n)_{n\ge j}$ is non-increasing. Thus, $f_k$ and all its
   derivatives $f_k', f_k'',\ldots$ are analytic functions on $D$.
%Note that $f_k^{(l)}(0)=0$ for $l\in\{1,\ldots,j-2\}$ and that
%$f_k^{(j-1)}(0)=(j-1)!\me(\tau_{j,1}^{k_1}\cdots\tau_{j,j}^{k_j})$.
% We shall see soon, that the functions $f_k$ are analytic in the
% domain $D:=\{s\in\cz\,:\,|s|<1\}$. The functions $f_k$ even have
% analytic continuations on $\cz\setminus [1,\infty)$, but we will
%not use these continuations in our approach.
   In order to state the following result it is convenient to introduce
   $L(z):=-\log(1-z)$, $z\in D$, and to define the functions $g_k:D\to\cz$,
   $k=(k_1,\ldots,k_j)\in\nz^j$, via $g_1(z):=z/(1-z)$ and
   \begin{equation} \label{gk}
      g_k(z)\ :=\ \sum_{i=1}^j k_i f_{k-e_i}^{(j-1)}(z)
   \end{equation}
   for all $z\in D$ and all $k=(k_1,\ldots,k_j)\in\nz^j$ satisfying $k_1+\cdots+k_j>1$,
   where $e_i$, $i\in\{1,\ldots,j\}$, denotes the $i$th unit vector in $\rz^j$.
\begin{lemma}
   For the Bolthausen--Sznitman coalescent,
   the function $f_k$, $k=(k_1,\ldots,k_j)\in\nz^j$, satisfies the
   differential equation
   \begin{equation} \label{mydgl}
      \frac{{\rm d}}{{\rm d}z}\big((L(z))^{j-1}f_k^{(j-1)}(z)\big)
      \ =\ \frac{(L(z))^{j-2}}{1-z}g_k(z),
      \quad z\in D\setminus\{0\},
   \end{equation}
   with solution
   \begin{equation} \label{fksolution}
      f_k^{(j-1)}(z)\ =\
      \frac{1}{(L(z))^{j-1}}\int_0^z \frac{(L(t))^{j-2}}{1-t}g_k(t)\,{\rm d}t,
      \quad z\in D\setminus\{0\}.
   \end{equation}
%   where $L(z):=-\log(1-z)$ and the functions $g_k:D\to\cz$, $k=(k_1,\ldots,k_j)\in\nz^j$,
%   are recursively defined via $g_1(z):=z/(1-z)$ and
%   $g_k(z):=\sum_{i=1}^j k_i f_{k-e_i}^{(j-1)}(z)$ for $k=(k_1,\ldots,k_j)\in
%   \nz^j$ with $k_1+\cdots+k_j>1$. Here $e_i$ denotes the $i$th unit
%   vector in $\rz^j$.
   In particular,
   \begin{equation} \label{particular}
      f_1(z)\ =\ \int_0^z \frac{t}{(1-t)^2L(t)}\,{\rm d}t
      \quad\mbox{and}\quad
      f_{(1,1)}'(z)\ =\ \frac{2}{L(z)}\int_0^z\frac{t}{(1-t)^3L(t)}\,{\rm d}t,
      \quad z\in D\setminus\{0\}.
   \end{equation}
\end{lemma}

\begin{proof}
   For the Bolthausen--Sznitman coalescent, $g_{nm}=n/((n-m)(n-m+1))$,
   $m\in\{1,\ldots,n-1\}$ and $g_n=n-1$, $n\in\nz$. Thus, $p_{nm}:=
   g_{nm}/g_n=n/((n-1)(n-m)(n-m+1))$, $m,n\in\nz$ with $m<n$.
   Fix $j\in\nz$ and $k=(k_1,\ldots,k_j)\in \nz^j$ and, for $n\in\nz$, define
   $a_n:=\mu_n(k)$ for convenience. For $n\ge\max(2,j)$ the recursion
   (\ref{generalrec}) reads
   $$
   a_n
   \ = \ q_n + \sum_{m=j+1}^{n-1}p_{nm}\frac{(m-1)_j}{(n)_j}\,a_m
%      & = & q_n + \sum_{m=1}^{n-1}\frac{n}{(n-1)(n-m)(n-m+1)}
%            \frac{(m-1)_j}{(n)_j}\,a_m\\
   \ = \ q_n + \frac{n}{(n-1)(n)_j}\sum_{m=j+1}^{n-1}
         \frac{(m-1)_j}{(n-m)(n-m+1)}\,a_m,
   $$
   where $q_n:=g_n^{-1}\sum_{i=1}^j k_i\mu_n(k-e_i)$ for all $n\ge\max(2,j)$.
   Thus,
   \begin{equation} \label{localrec}
      (n-1)(n-1)_{j-1}a_n\ =\ (n-1)(n-1)_{j-1}q_n + \sum_{m=j+1}^{n-1}
      \frac{(m-1)_j}{(n-m)(n-m+1)}a_m.
   \end{equation}
%   Note that this equation holds not only for $n\ge j$, but for all $n\ge 2$.
   Before we come back to the recursion (\ref{localrec}) let us
   first verify that
   \begin{equation} \label{gk2}
      \sum_{n=\max(2,j)}^\infty (n-1)(n-1)_{j-1}q_nz^{n-j}\ =\ g_k(z),
      \qquad z\in D.
   \end{equation}
   Obviously (\ref{gk2}) holds for $j=1$ and $k_1=1$, since
   in this case $q_n=1/g_n=1/(n-1)$ and $g_1(z)=z/(1-z)$ by definition.
   For $k=(k_1,\ldots,k_j)\in\nz^j$ with $k_1+\cdots+k_j>1$ we have
   \begin{eqnarray*}
      &   & \hspace{-25mm}
            \sum_{n=\max(2,j)}^\infty (n-1)(n-1)_{j-1}q_nz^{n-j}
      \ = \ \sum_{n=\max(2,j)}^\infty (n-1)_{j-1}\sum_{i=1}^j k_i
            \mu_n(k-e_i) z^{n-j}\\
      & = & \Big(\frac{{\rm d}}{{\rm d}z}\Big)^{j-1}\sum_{i=1}^j k_i\sum_{n=\max(2,j)}^\infty
            \mu_n(k-e_i) z^{n-1}
      \ = \ \sum_{i=1}^j k_i f_{k-e_i}^{(j-1)}(z)
      \ = \ g_k(z).
   \end{eqnarray*}
   Thus, (\ref{gk2}) is established.
   In view of $(n-1)(n-1)_{j-1}=(n-1)_j+(j-1)(n-1)_{j-1}$
   and (\ref{gk2}), by multiplying both sides in
   (\ref{localrec}) with $z^{n-j}$
   and summing over all $n\ge \max(2,j)$, the
   recursion (\ref{localrec}) translates to
   \begin{eqnarray}
      zf_k^{(j)}(z) + (j-1)f_k^{(j-1)}(z)
      & = & g_k(z) + \sum_{n=\max(2,j)}^\infty \sum_{m=j+1}^{n-1}\frac{(m-1)_j}{(n-m)(n-m+1)} a_mz^{n-j}\nonumber\\
      & = & g_k(z) + \sum_{m=j+1}^\infty (m-1)_ja_mz^{m-j}\sum_{n=m+1}^\infty \frac{1}{(n-m)(n-m+1)}z^{n-m}\nonumber\\
      & = & g_k(z) + za(z)\Big(\frac{{\rm d}}{{\rm d}z}\Big)^j
            \sum_{m=j}^\infty a_m z^{m-1}\nonumber\\
      & = & g_k(z) + za(z)f_k^{(j)}(z), \label{ode}
   \end{eqnarray}
   where $a(z):=\sum_{n=1}^\infty z^n/(n(n+1))$ for $z\in D$.
   Since $z(1-a(z))=(1-z)L(z)$, the differential
   equation (\ref{ode}) can be rewritten in the form (\ref{mydgl}).
   For $j>1$ the only solution of (\ref{mydgl}) being continuous at $0$ (and for
   $j=1$ the only solution of (\ref{mydgl}) with $f_k(0)=0$) is given by (\ref{fksolution}).
   Since $g_1(z)=z/(1-z)$, (\ref{fksolution}) reduces
   for $j:=k_1:=1$ to the first
   equation in (\ref{particular}),
   in agreement with \cite[Lemma 3.1, Eq.~(3.3)]{freundmoehle}).
   Noting that $g_{(1,1)}(z)=f_{(0,1)}'(z)+f_{(1,0)}'(z)=2f_1'(z)=2z/((1-z)^2L(z))$,
%   $$
%   g_{(1,1)}(z)=\sum_{n=2}^\infty (n-1)^2q_nz^{n-2}
%%   =\sum_{n=2}^\infty (n-1)^2 (2/(n-1)\me(\tau_{n,1}))z^{n-2}
%   =2\sum_{n=2}^\infty (n-1)\me(\tau_{n,1})z^{n-2}=2f_1'(z)=
%   \frac{2z}{(1-z)^2(-\log(1-z))},
%   $$
   the formula for $f_{(1,1)}'(z)$ in (\ref{particular}) follows by choosing
   $j:=2$ and $k_1:=k_2:=1$ in
   (\ref{fksolution}).\hfill$\Box$
\end{proof}
\begin{corollary} {\bf (Exact formula for $\me(\tau_{n,1}\tau_{n,2})$)}
\label{corollary32}
   \ \\
   Fix $n\in\{2,3,\ldots\}$. For the Bolthausen--Sznitman coalescent,
   \begin{equation} \label{tau12exact}
      \me(\tau_{n,1}\tau_{n,2})
      \ =\ \frac{2}{(n-1)!}\sum_{k=1}^{n-1}\frac{2^k-1}{k^2}s(n-2,k-1),
   \end{equation}
   where the $s(n,k)$ denote the absolute Stirling numbers of the
   first kind.
\end{corollary}
\begin{remark}
   Together with the exact formula
   $\me(\tau_{n,1})=((n-1)!)^{-1}\sum_{k=1}^{n-1}s(n-1,k)/k$ for
   the mean of $\tau_{n,1}$ (see, for example, Proposition 1.2 of \cite{freundmoehle})
   it can be
   checked that ${\rm Cov}(\tau_{n,1},\tau_{n,2})=
   \me(\tau_{n,1}\tau_{n,2})-(\me(\tau_{n,1}))^2>0$ for all
   $n\ge 2$. Thus, for all $n\ge 2$, $\tau_{n,1}$ and $\tau_{n,2}$
   are positively correlated. We have used the exact formulas for
   $\me(\tau_{n,1})$ and $\me(\tau_{n,1}\tau_{n,2})$ to compute
   the entries of the following table.
   \begin{center}
      \begin{tabular}{|r|l|l|l|}
         \hline
         $n$      & $\me(\tau_{n,1})$ & $\me(\tau_{n,1}\tau_{n,2})$ & ${\rm Cov}(\tau_{n,1},\tau_{n,2})$\\
         \hline
         $2$      & $1$ & $2$ & $1$\\
         $3$      & $0.75$ & $0.75$ & $0.1875$\\
         $4$      & $0.638889$ & $0.509259$ & $0.101080$\\
         $5$      & $0.572917$ & $0.397569$ & $0.069336$\\
         $10$     & $0.431647$ & $0.215119$ & $0.028800$\\
         $100$    & $0.228368$ & $0.057067$ & $0.004915$\\
%         $200$    & $0.198537$ & $0.042758$ & $0.003341$\\
%         $300$    & $0.184283$ & $0.036676$ & $0.002716$\\
%         $400$    & $0.175300$ & $0.033092$ & $0.002362$\\
%         $500$    & $0.168891$ & $0.030652$ & $0.002128$\\
%         $1000$   & $0.151582$ & $0.024546$ & $0.001568$\\
%         \hline
%         $n\to\infty$ & $\sim\frac{1}{\log n}$ & $\sim\frac{1}{\log^2n}$ & \\
         \hline
      \end{tabular}\\

      \vspace{1mm}

      Table 2: Covariance of $\tau_{n,1}$ and $\tau_{n,2}$
      for the Bolthausen--Sznitman coalescent
   \end{center}
\end{remark}
\begin{proof} (of Corollary \ref{corollary32})
   We write $f:=f_{(1,1)}$ for convenience.
   The substitution $u=L(t)=-\log(1-t)$ below the second integral
   in (\ref{particular}) yields
   \begin{eqnarray*}
      f'(z)
      & = & \frac{2}{L(z)}\int_0^{L(z)}\frac{e^{2u}-e^u}{u}\,{\rm d}u\\
      & = & \frac{2}{L(z)}\int_0^{L(z)}\frac{1}{u}
            \bigg(\sum_{k=0}^\infty\frac{(2u)^k}{k!}- \sum_{k=0}^\infty\frac{u^k}{k!}\bigg)\,{\rm d}u\\
      & = & \frac{2}{L(z)} \sum_{k=1}^\infty \frac{2^k-1}{k!}
            \int_0^{L(z)} u^{k-1}\,{\rm d}u\\
      & = & \frac{2}{L(z)}\sum_{k=1}^\infty \frac{2^k-1}{k!}\frac{{(L(z))^k}}{k}\\
      & = & 2\sum_{k=1}^\infty \frac{2^k-1}{k k!}(L(z))^{k-1}.
   \end{eqnarray*}
   From (see \cite[p.~824]{abramowitzstegun})
   $(L(z))^k/k!=\sum_{i=k}^\infty z^i/i!s(i,k)$ we conclude that
   $$
   f'(z)
   \ =\ 2\sum_{k=1}^\infty \frac{2^k-1}{k^2}\sum_{i=k-1}^\infty \frac{z^i}{i!}s(i,k-1)
   \ =\ 2\sum_{i=0}^\infty \frac{z^i}{i!}\sum_{k=1}^{i+1}\frac{2^k-1}{k^2}s(i,k-1).
   $$
   For a power series $g(z)=\sum_{n=0}^\infty g_nz^n$ we denote in the following
   with $[z^n]g(z):=g_n$ the coefficient in front of $z^n$ in the series
   expansion of $g$. Using this notation we obtain
   $$
   (i+1)\me(\tau_{i+2,1}\tau_{i+2,2})
   \ =\ [z^i]f'(z)
   \ =\ \frac{2}{i!}\sum_{k=1}^{i+1}\frac{2^k-1}{k^2}s(i,k-1).
   $$
   It remains to divide by $i+1$ and to substitute $n=i+2$.\hfill$\Box$
\end{proof}
\subsection{Proofs of Theorem \ref{theorem12}, Corollary \ref{corollary13},
and Corollary \ref{corollary14}} \label{furtherproofs}
\setcounter{theorem}{0}
\begin{proof} (of Theorem \ref{theorem12})
   Let us verify (\ref{bsasy}) by induction on the degree $d:=k_1+\cdots+k_j$.
   Obviously (\ref{bsasy}) holds for $d=0$, i.e. for all
   $j\in\nz$ and $k_1=\cdots=k_j=0$. In order to verify (\ref{bsasy})
   for $d=1$ it suffices to show that $a_n:=\mu_n(1)\sim 1/\log n$ as
   $n\to\infty$ since $\mu_n(k)=\mu_n(1)$ for all
   $k=(k_1,\ldots,k_j)\in\nz_0^j$ satisfying $k_1+\cdots+k_j=1$.
   By (\ref{particular}) and de l'Hospital's rule
   $$
   f_1(z)\ =\ \int_0^z\frac{t}{(1-t)^2L(t)}\,{\rm d}t
   \ \sim\ \frac{1}{(1-z)L(z)},\qquad z\nearrow 1.
   $$
   Since $a_n=\me(\tau_{n,1})$ is non-increasing in $n$,
   Karamata's Tauberian theorem for power series
   \cite[p.~40, Corollary 1.7.3]{binghamgoldieteugels}, applied with
   $c:=\rho:=1$ and $l(x):=1/\log x$, yields $a_n\sim l(n)=1/\log n$.
   Thus, (\ref{bsasy}) holds for $d=1$.

   In order to verify the induction step from $d-1$ to $d>1$ fix
   $k=(k_1,\ldots,k_j)\in\nz_0^j$ with $d:=k_1+\cdots+k_j>1$.
%   Due to the definition of the coefficients
%   $\kappa_j(k)$ in Theorem \ref{theorem12} for the case when $k_i=0$ for
%   some $i$
   We can and do assume without loss of generality that
   $k=(k_1,\ldots,k_j)\in\nz^j$. By the induction hypothesis
   $$
   b_n\ :=\ \sum_{i=1}^j k_i\mu_n(k-e_i)
   \ \sim\ \sum_{i=1}^j k_i\frac{k_1!\cdots (k_i-1)!\cdots k_j!}{\log^{d-1}n}
   \ \sim\ \frac{jk_1!\cdots k_j!}{\log^{d-1}n},\qquad n\to\infty.
   $$
   Since $b_n$ is non-increasing in $n$, the same Tauberian theorem as
   used above for $d=1$, but now applied
%   \cite[p.~40, Corollary 1.7.3]{binghamgoldieteugels}
%   (applied
   with $c:=jk_1!\cdots k_j!$, $\rho:=1$ and $l(x):=1/\log^{d-1}x$,
   yields
   $$
   b(z)\ :=\ \sum_{n=\max(2,j)}^\infty b_n z^{n-1}
   \ \sim\ \frac{jk_1!\cdots k_j!}{(1-z)(L(z))^{d-1}},\qquad z\nearrow 1.
   $$
   Note that $b(z)=\sum_{i=1}^j k_if_{k-e_i}(z)$.
   Applying de l'Hospital's rule $(j-1)$-times yields
   $$
   g_k(z)\ =\ \sum_{i=1}^j k_if_{k-e_i}^{(j-1)}(z)\ =\ b^{(j-1)}(z)
   \ \sim\ \frac{j!k_1!\cdots k_j!}{(1-z)^j(L(z))^{d-1}},\qquad z\nearrow 1.
   $$
%   From (\ref{gk2}) and the (simple) Tauberian theorem
%   $$
%   g_k(z)
%   \ =\ \sum_{n=\max(2,j)}^\infty (n-1)(n-1)_{j-1}q_nz^{n-j}
%   \ \sim\ \frac{j!k_1!\cdots k_j!}{(1-z)^j(L(z))^{d-1}},
%   \quad z\nearrow 1
%   $$
   Thus, by (\ref{fksolution}) and by one further application of
   de l'Hospital's rule,
   $$
   f_k^{(j-1)}(z)
   \ =\ \frac{1}{(L(z))^{j-1}}\int_0^z \frac{(L(t))^{j-2}}{1-t}g_k(t)\,{\rm d}t
   \ \sim\ \frac{(j-1)!k_1!\cdots k_j!}{(1-z)^j(L(z))^d},
   \qquad z\nearrow 1.
   $$
   Using again de l'Hospital's rule $(j-1)$-times it follows that
   $$
   f_k(z)\ \sim\ \frac{k_1!\cdots k_j!}{(1-z)(L(z))^d},\qquad z\nearrow 1.
   $$
   Since $a_n:=\mu_n(k)$ is non-increasing in $n$, again Karamata's Tauberian
   theorem for power series, now applied with $c:=k_1!\cdots k_j!$, $\rho:=1$
   and $l(x):=1/\log^dx$, yields $a_n\sim k_1!\cdots k_j!/\log^dn$.\hfill$\Box$
\end{proof}
\begin{proof} (of Corollary \ref{corollary13})
   Theorem \ref{theorem12} clearly implies that, for $j\in\nz$ and
   $k_1,\ldots,k_j\in\nz_0$,
   $\me((\tau_{n,1}\log n)^{k_1}\cdots(\tau_{n,j}\log n)^{k_j})=
   (\log n)^{k_1+\cdots+k_j}\mu_n(k_1,\ldots,k_j)\to k_1!\cdots k_j!
   =\me(\tau_1^{k_1}\cdots\tau_j^{k_j})$ as $n\to\infty$.
   For all $i\in\{1,\ldots,j\}$ and all $0\le\theta<1$ we have
   $\sum_{r=0}^\infty (\theta^r/r!)\me(\tau_i^r)=\sum_{r=0}^\infty \theta^r
   =1/(1-\theta)<\infty$. Therefore (see \cite{billingsley1}, Theorems 30.1
   and 30.2 for the one-dimensional case and Problem 30.6 on p.~398 for the
   multi-dimensional case) the above convergence of moments implies the
   convergence $(\log n)(\tau_{n,1},\ldots,\tau_{n,j})\to
   (\tau_1,\ldots,\tau_j)$ in distribution as $n\to\infty$ for each $j\in\nz$.
   The convergence of all these $j$-dimensional distributions is already
   equivalent (see Billingsley \cite[p.~19]{billingsley2}) to the convergence
   of the full processes $(\log n)(\tau_{n,1},\ldots,\tau_{n,n},0,0,\ldots)
   \to(\tau_1,\tau_2,\ldots)$ in distribution as $n\to\infty$.\hfill$\Box$
\end{proof}
\begin{proof} (of Corollary \ref{corollary14})
   The external branch length $L_n^{external}$
   satisfies (see \cite[p.~2165]{moehleproper})
   $$
   \me((L_n^{external})^k)
   \ =\ \sum_{j=1}^k {n\choose j}
   \sum_{{k_1,\ldots,k_j\in\nz}\atop{k_1+\cdots+k_j=k}}
   \frac{k!}{k_1!\cdots k_j!}\,\mu_n(k_1,\ldots,k_j),
   \qquad n\in\{2,3,\ldots\}, k\in\nz.
   $$
   By Theorem \ref{theorem12},
   $\mu_n(k_1,\ldots,k_j)\sim k_1!\cdots k_j!/\log^kn$ as $n\to\infty$.
   Therefore, asymptotically
   the summand with index $j=k$ dominates the others,
   so asymptotically
   all the summands with indices $j<k$ can be disregarded.
   Thus, $\me((L_n^{external})^k)\sim {n\choose k}k!/\log^kn
   \sim n^k/\log^kn$. This convergence of all moments
   $\me((\frac{\log n}{n}L_n^{external})^k)\to 1$ as $n\to\infty$
   implies the convergence $\frac{\log n}{n}L_n^{external}\to 1$ in distribution
   (and hence in probability) as $n\to\infty$.\hfill$\Box$
\end{proof}
\begin{acknowledgement}
   The first author benefited from the support of the `Agence Nationale de
   la Recherche': ANR MANEGE (ANR09-BLAN-0215). Research was partially
   completed while the second author was visiting the Institute for
   Mathematical Sciences, National University of Singapore in 2011.
%
%   When you have written your paper(s), please send us:
%
%   (a) a soft copy of your preprint so that we can include it in our IMS
%       Pre-print Series;
%   (b) a soft or hard copy of your reprint for our collection when the
%       paper is published.
%
   The second author furthermore acknowledges the support of the Universit\'e
   Paris 13. He in particular thanks the Laboratoire Analyse, G\'eometrie et
   Applications at the Universit\'e Paris 13 for its kind hospitality during
   a one month visit in 2011, where parts of this work was done. The authors
   are indebted to an anonymous referee for his profound suggestions leading
   to a significant improvement of the article.
\end{acknowledgement}

\end{document}